\documentclass[10pt,a4paper]{amsart}
\usepackage[latin1]{inputenc}
\usepackage{amsmath}
\usepackage{amsfonts}
\usepackage{amssymb}
\usepackage{color}
\usepackage[all]{xy}
\usepackage{amscd}

\usepackage{color}

\newtheorem{theorem}{Theorem}[section]
\newtheorem{lemma}{Lemma}[section]

\newtheorem{corollary}{Corollary}[section]
\newtheorem{defn}{Definition}[section]
\newtheorem{remark}{Remark}[section]
\newtheorem{claim}{Claim}[section]
\newtheorem{prop}{Proposition}[section]

\newtheorem{conjecture}{Conjecture}[section]

\address{Division of Mathematics, Graduate School of Science,
         Hokkaido University,  Sapporo, 060-0810 Japan}
\thanks{* Partially supported by Grand-in-Aid \# 18684001
 (Japan Society for Promortion of Sciences).} 
\email{matusita@math.sci.hokudai.ac.jp}

\author{Daisuke MATSUSHITA}
\title{On isotropic divisors on irreducible symplectic manifolds}
\begin{document}
\maketitle

\begin{abstract}
Let $ X $ be an irreducible symplectic manifold and $ L $ a divisor on $ X $. Assume that
$ L $ is isotropic with respect to the Beauville-Bogomolov quadratic form.
We define the rational Lagrangian locus and the movable locus 
on the universal deformation space of the pair $ (X,L) $.
We prove that the rational Lagrangian locus is 
empty or coincide with the movable locus of the universal deformation space.
\end{abstract}

\section{Introduction}

We start with recalling the definition of an irreducible symplectic manifold.

\begin{defn}[{{\cite[Th\'{e}or\`{e}m 1]{Beauville}}}]
Let $ X $ be a compact K\"ahler manifold. The manifold $ X $ is said be irreducible symplectic
if $ X $ satisfies the following three properties.
\begin{itemize}
\item[(1)] $ X $ carries a symplectic form.
\item[(2)] $ X $ is simply connected.
\item[(3)] $ \dim H^{0}(X,\Omega_{X}^{2}) = 1$.
\end{itemize}
\end{defn}

Together with  Calabi-Yau manifolds and complex tori, irreducible symplectic manifolds
form a building block of a compact K\"ahler manifold with $ c_{1} = 0 $.
It is shown in \cite{fibrespace}, \cite{addendum_fibrespace} and \cite{MR2453602}
 that a fibre space structure of an irreducible symplectic manifold
is very restricted. To state the result, we recall the definition of a Lagrangian fibration.

\begin{defn}\label{rational_Lag_def}
Let $X$ be an irreducible symplectic  manifold and $L$
a line bundle on $X$.  A surjective morphism
$ g  : X \to S$ is said to be Lagrangian
if a general fibre is connected and Lagrangian. A dominant map
$ g: X \dashrightarrow S $ is said to be rational Lagrangian
if there exists a birational map $\phi : X\dasharrow X'$ such that
the composite map $g\circ \phi^{-1} : X' \to S$
is Lagrangian.
We say that $L$ defines a 
Lagrangian fibration if the linear system $ |L| $ defines a Lagrangian fibration.
Also we  say that $L$ defines a rational Lagrangian fibration if $|L|$ defines a rational Lagrangian fibration.
\end{defn}

\begin{theorem}[\cite{addendum_fibrespace}, \cite{fibrespace} and \cite{MR2453602}]
Let $ X $ be a projective irreducible symplectic manifold. Assume that $ X $ admits a 
surjective morphism $ g : X \to S $ over a smooth projective manifold $ S $. 
Assume that $ 0 < \dim S < \dim X $ and $ g $ has connected fibres.
Then
$ g $ is Lagrangian and $ S \cong \mathbb{P}^{1/2\dim X}$.
\end{theorem}

It is a natural question when a line bundle $ L $ defines a Lagrangian fibration. If
$ L $ defines a rational Lagrangian fibration, then $ L $ is isotropic with repect to
the Beauville-Bogomolov quadratic form. Moreover the first Chern class $ c_{1}(L) $ of $ L $ 
belongs to the biational K\"ahler cone which is defined in \cite[Definition 4.1]{MR1992275}.

\begin{conjecture}[D.~Huybrechts and J.~Sawon]\label{conjecture}
Let $ X $ be an irreducible symplectic manifold and $ L $  a line bundle on $ X $. 
Assume that $ L $ is
 isotropic with respect to the Beauville-Bogomolov quadratic form on $ H^{2}(X,\mathbb{C}) $.
 We also assume that $ c_{1}(L) $ belongs to the birational K\"ahler cone of $ X $.
Then $ L $ will define a rational Lagrangian fibration.
\end{conjecture}

At that moment, partial results are known about Conjecture \ref{conjecture}. 
We could consult \cite{MR2400885}, \cite{MR2739808},
\cite{MR2357635} and \cite{MR2585581}. In this note,
we consider the above conjecture by a different approach. To state the result,
we recall the basic facts of a deformation of a pair which consists of a symplectic manifold and
a line bundle.

\begin{defn}
 Let $X$ be a K\"{a}hler manifold and $L$ a line bundle on $X$.
 A deformation of the pair $(X,L)$ consists of a smooth
 morphism $\mathfrak{X}\to S$ over a smooth manifold $S$
 with a reference point $o$ and a line
 bundle $\mathfrak{L}$ on $\mathfrak{X}$ 
 such that the fibre $\mathfrak{X}_{o}$ at $o$ is isomorphic to $X$
 and the restriction $\mathfrak{L}|_{\mathfrak{X}_{0}}$ is 
 isomorphic to $L$. 
\end{defn}

If $ X $ is an irreducible symplectic manifold, it is known that there exists the universal 
deformation of deformations of a pair $ (X,L) $.

\begin{prop}[{{\cite[(1.14)]{MR1664696}}}]
Let $X$ be an irreducible symplectic manifold and
$L$ a line bundle on $X$. We also let $\mathfrak{X} \to \mathrm{Def}(X)$
be the Kuranishi family of $X$. Then there exists a smooth hypersurface 
$\mathrm{Def}(X,L)$ of $\mathrm{Def}(X)$ such that the restriction family
$\mathfrak{X}_{L}:= \mathfrak{X}\times_{\mathrm{Def}(X)}\mathrm{Def}(X,L) 
\to \mathrm{Def}(X,L)$ forms the universal family of deformations of the pair
$(X,L)$. Namely, $\mathfrak{X}_{L}$ carries a line bundle $\mathfrak{L}$
and every deformation $\mathfrak{X}_{S} \to S$ of $(X,L)$
is isomorphic to the pull back of $(\mathfrak{X}_{L},\mathfrak{L})$
via a uniquely determined map
$S \to \mathrm{Def}(X,L)$.
\end{prop}

Now  we can state the result.

\begin{theorem}\label{main_result}
Let $X$ be an irreducible symplectic manifold
and $L$ a line bundle on $X$. We also let $\pi : \mathfrak{X}_{L} \to \mathrm{Def}(X,L)$
be the universal family of deformations of the pair $(X,L)$
and $\mathfrak{L}$ the universal bundle. We denote by $q$
the Beauville-Bogomolov form on $H^{2}(X,\mathbb{C})$.
Assume that $q(L)=0$. We define the locus of movable $ \mathrm{Def}(X,L)_{\mathrm{mov}} $ by
\[
 \{
 t \in \mathrm{Def}(X,L);\mbox{$c_{1}(\mathfrak{L}_{t})$ belongs to the birational K\"ahler cone of $ X $.}
 \}
\]
We also define more two subsets of $ \mathrm{Def}(X,L) $. The first is
the
locus of rational Lagrangian fibration $V$ which is defined by
\[
 \{
 t \in \mathrm{Def}(X,L);\mbox{$\mathfrak{L}_{t}$ defines a rational Lagrangian 
 fibration over the projective space. }
 \}
\]
The second is the locus of Lagrangian fibration $V_{\mathrm{reg}}$ which is defined by
\[
 \{
 t \in \mathrm{Def}(X,L);\mbox{$\mathfrak{L}_{t}$ defines a Lagrangian
 fibration over the projective space. }
 \}
\]
Then $V=\emptyset$ or $V=\mathrm{Def}(X,L)_{\mathrm{mov}}$. Moreover if $V\ne \emptyset$, $V_{\mathrm{reg}}$ 
is a dense open subset of $\mathrm{Def}(X,L)$ and
$ \mathrm{Def}(X,L)\setminus V_{\mathrm{reg}} $ is contained in
a union of countably hypersurfaces of $ \mathrm{Def}(X,L) $.
\end{theorem}

\noindent

\begin{remark}
Professors L.~Kamenova and M.~Verbitsky obtained $ V_{\mathrm{reg}} $ is an dense open set of
$ \mathrm{Def}(X,L) $ under the assumption $ V_{\mathrm{reg}} \ne \emptyset $ in \cite[Theorem 3.4]{1208.4626}.
\end{remark}

 To state an application of Theorem \ref{main_result}, we need the following two definitions.
\begin{defn}
Two compact K\"ahler manifolds $ X $ and $ X' $ are 
said to be deformation equivalent if there exists a proper smooth
morphism $\pi : \mathfrak{X} \to S $ over a smooth connected complex manifold $ S $
such that both $ X $ and $ X' $ form fibres of $ \pi $.
\end{defn}

\begin{defn}
An irreducible symplectic manifold $ X $ is said to be of $ K3^{[n]} $-type if 
$ X $ is deformation equivalent to the $ n $-pointed Hilbert scheme of a $ K3 $ surface.
An irreducible symplectic manifold $ X $
is also said to be of type generalized Kummer 
if $ X $ is deformation
equivalent to
a generalized Kummer variety which
is defined in \cite[Th\'{e}or\`{e}m 4]{MR785234}.
\end{defn}

\noindent
It was shown in
 \cite{1301.6584}, \cite{1301.6968} and \cite{1206.4838}
 that if $ X $ is isomorphic to the $ n $-pointed Hilbert scheme of a $ K3 $ surface
 or a generalized Kummer variety, then Conjecture \ref{conjecture} holds.
To combine these results and Theorem \ref{main_result}, we obtain the following result.

\begin{corollary}\label{application}
Let $ X $ be an irreducible symplectic manifold of type $ K3^{[n]} $
or of type generalized Kummer. 
We also let $ L $
 be a line bundle $ L $ on $ X $ which is not trivial,  isotropic
with respect to the Beauville-Bogomolov quadratic form on $ H^{2}(X,\mathbb{C}) $and
$ c_{1}(L) $ belongs to the birational K\"ahler cone of $ X $.
 Then $ L $ define a
rational Lagrangian fibration
over the projective space.
\end{corollary}

\section*{Acknowledgement} The author would like to express his gratitude to Professors
Markushevich Dimitri,
Keiji Oguiso and Misha Vervitsky
 for their comments on the earlier version of this work. 
He also would like to express his gratituede to Professors
Arend Bayer, Emanuele Macr\`i, Eyal Markman and Kota Yoshioka for sending me their papers.

\section{Birational correspondence of deformation families}
In this section we study a relationship between 
deformation families. We start with introducing the following Lemma.
 
\begin{lemma}[{{\cite[Lemma 2.6]{MR1664696}}}]\label{isometry}
Let $ X $ and $ X' $ be irreducible symplectic manifolds. Assume that there 
exists a bimeromorphic map $ \phi : X \dasharrow X' $.
Then $ \phi $ induces an isomorphism
\[
 \phi_{*} : H^{2}(X,\mathbb{C}) \cong H^{2}(X',\mathbb{C})
\]
which compatible with the Hodge structures and the Beauville-Bogomolov quadratic forms.
\end{lemma}

We consider the relationship of the Kuranishi families 
of bimeromorphic irreducible symplectic manifolds.
\begin{prop}\label{birational_kuranishi}
  Let $X$ and $X'$ are irreducible symplectic manifolds. We denote
by $\pi : \mathfrak{X}\to\mathrm{Def}(X)$ the universal family of  deformations
of $X$. We also denote by $\pi' : \mathfrak{X}' \to \mathrm{Def}(X')$ the universal family of
 deformations of $X'$. Assume that $X$ and $X'$ are bimeromorphic.
Then there exist dense open subset $U$ of $\mathrm{Def}(X)$ 
and $ U' $ of $ \mathrm{Def}(X') $ which satisfy the following three properties.
\begin{itemize}
\item[(1)] The set $ \mathrm{Def}(X)\setminus U $ is contained in 
a union of countably hypersurfaces in $ \mathrm{Def}(X) $ and 
$ \mathrm{Def}(X')\setminus U'$ is also contained in a union
of countably hypersurfaces in $ \mathrm{Def}(X') $.
\item[(2)] They satisfy the following diagram:
\begin{equation*}\label{birational_diagram}
 \xymatrix{
 \mathfrak{X}\times_{\mathrm{Def}(X)}U
  \ar[r]^{\cong}_{\tilde{\phi}} \ar[d] & \mathfrak{X}'\times_{\mathrm{Def}(X')}{U'} 
   \ar[d] \\
 U \ar[r]_{\cong}^{\varphi} & U' ,
 } 
\end{equation*}
where $ \tilde{\phi} $ and $ \varphi $ are isomorphic.
\item[(3)] Let $ s $ be a point of $ U $ and $ s' $ the point $ \varphi (s) $. We also let
$ \phi_{s} : \mathfrak{X}_{s} \cong \mathfrak{X}'_{s'
} $ be 
the restriction of the isomorphism 
$\tilde{\phi} : \mathfrak{X}\times_{\mathrm{Def}(X)} {U} \cong \mathfrak{X}' \times_{\mathrm{Def}(X')} U'$ 
in the above diagram to the fibre $ \mathfrak{X}_{s} $ at $s$
and the fibre $ \mathfrak{X}'_{s'} $ at $ s' $. 
We denote by $ \eta $  a parallel transport in the local system
$ R^{2}\pi_{*} \mathbb{C}$ along a path from the reference point to $ s $.
We also denote by  $ \eta' $ a parallel transport in the local system
$ R^{2}\pi'_{*} \mathbb{C}$ along a path from the reference point to $ s' $.
Then the composition of the isomorphisms
\[	
 H^{2}(X,\mathbb{C}) \stackrel{\eta}{\cong} H^{2}(\mathfrak{X}_{s},\mathbb{C}) 
 \stackrel{\phi_{s}}{\cong} H^{2}(\mathfrak{X}'_{s'}, \mathbb{C})
 \stackrel{\eta'^{-1}}{\cong} H^{2}(X', \mathbb{C})
\]
coincides with $ \phi_{*} $ which is the isomorphism induced by
$ \phi : X \dasharrow X' $.
\end{itemize}
\end{prop}

\begin{proof}
  The proof of this proposition is a mimic of the proof of 
  \cite[Theorem 5.9]{MR1664696}.
  The proof consists of two steps. First, we show that there exist
  open sets $ U $ of $ \mathrm{Def}(X) $ and $ U' $ of $ \mathrm{Def}(X') $
  which satisfy the the assertions (2) and (3) of Proposition \ref{birational_kuranishi}.
Since $ X $ and $ X' $ are bimeromorphic,
we have a deformation $\mathfrak{X}_{S}\to S$
of $X$ and a deformation $\mathfrak{X}'_{S}\to S$ of $X'$ over
a small disk $S$ which are
 isomorphic to each other over the
punctured disk $S\setminus0$ by \cite[Theorem 4.6]{MR1664696}.
By the universality, 
 $ \mathfrak{X}_{S} \to S$ is isomorphic 
to the base change $ \mathfrak{X} \to \mathrm{Def}(X) $ 
by a uniquely determined morphism
$ S \to \mathrm{Def}(X) $. The family $ \mathfrak{X}'_{S} \to S $ is also
isomorphic to the base change of 
$ \mathfrak{X}' \to \mathrm{Def}(X') $
by a uniquely determined morphism
$ S \to \mathrm{Def}(X') $. 
Thus
 there exist  points $t \in\mathrm{Def}(X)$ and 
 $ t'  \in \mathrm{Def}(X') $
such that the fibres $\mathfrak{X}_{t}$ and $\mathfrak{X}'_{t'}$
are isomorphic. 
Let $ \eta $ be a parallel transportation of $ R^{2}\pi_{*}\mathbb{C} $ along
a path from the reference point to $ t $ and
 $ \eta' $  a parallel transportation of $ R^{2}\pi'_{*}\mathbb{C} $ along
a path from the reference point to $ t' $.
To consider 
the composition of the isomorphisms
\begin{equation}\label{parallel_transport}
 H^{2}(X,\mathbb{C}) \stackrel{\eta}{\cong}
 H^{2}(\mathfrak{X}_{t},\mathbb{C}) 
 \cong H^{2}(\mathfrak{X}'_{t'},\mathbb{C}) 
 \stackrel{\eta'^{-1}}{\cong}
  H^{2}(X',\mathbb{C})
\end{equation}
we need more information of the construction of the two families $ \mathfrak{X}_{S} \to S $ and
$ \mathfrak{X}'_{S} \to S $. By the last paragraph of the proof of \cite[Theorem 4.6]{MR1664696},
the construction is due to \cite[Proposition 4.5]{MR1664696}. 
According to  \cite[Proposition 4.5]{MR1664696}, $ \mathfrak{X}'_{S} \to S $ is constructed as follows.
Let $ H' $ be an ample divisor on $ X' $ and $ H := (\phi_{*})^{-1}H'$. We consider a deformation 
$\pi_{S} :  (\mathfrak{X}_{S},\mathfrak{H}) \to S $ of $ (X,H) $. Then the closure of the image of the rational
map $ \mathfrak{X}_{S} \dasharrow \mathbb{P}((\pi_{S})_{*}\mathfrak{H}) $ gives the desired deformation
$ \mathfrak{X}'_{S} \to S $. 
Hence we have a birational map $ \mathfrak{X}_{S} \dasharrow \mathfrak{X}'_{S} $ which commutes with
the two projections. Moreover the restriction of this birational map
coincides with $ \phi $. Thus the composition of the isomorphisms (\ref{parallel_transport})
coincides with $ \phi_{*} $.
By \cite[Th\'{e}or\`{e}m 5 (b)]{Beauville}, we can extend the isomorphism
$ \mathfrak{X}_{t} \cong \mathfrak{X}'_{t'} $ over open sets of
$ \mathrm{Def}(X) $ and $ \mathrm{Def}(X') $, that is,
there exist open
sets $U$ of $\mathrm{Def}(X)$ and $ U' $ of $ \mathrm{Def}(X') $
such that the restriction families 
$\mathfrak{X} \times_{\mathrm{Def}(X)}U$
and $\mathfrak{X}'\times_{\mathrm{Def}(X')}U'$ are isomorphic and this isomorphism is compatible
with the two projections 
$\mathfrak{X}\to\mathrm{Def}(X)$ and $\mathfrak{X}'\to\mathrm{Def}(X')$.
By this construction, the restriction of the
isomorphism $ \tilde{\phi} : \mathfrak{X}\times_{\mathrm{Def}(X)}U \cong \mathfrak{X}'\times_{\mathrm{Def}(X')}U' $
to the fibres
satisfies the assertion (3) of Proposition \ref{birational_kuranishi}.

Next we show that $ U $ and $ U' $ satisfies the assertion (1) of 
Proposition \ref{birational_kuranishi}. 
Let $s$ be a point of $\bar{U}$. By \cite[Theorem 4.3]{MR1664696},
the fibres $\mathfrak{X}_{s}$ and $\mathfrak{X}'_{s}$ are bimeromorphic.
If $\dim H^{1,1}(\mathfrak{X}_{s},\mathbb{Q})=0$, 
then $\mathfrak{X}_{s}$ and $\mathfrak{X}'_{s}$
carries neither curves nor effective divisors. Thus $ \mathfrak{X}_{s} $ and $ \mathfrak{X}'_{s} $
are isomorphic by \cite[Proposition 2.1]{MR1992275} and  $ s \in U $.
Thus if $s \in \bar{U}\setminus U$ 
then $  \dim H^{1,1}(\mathfrak{X}_{s},\mathbb{C}) \ge 1$. This implies
that $ \bar{U}\setminus U $ is contained in a union of
countably hypersurfaces.
\end{proof}

For the proof of Theorem \ref{main_result},
we need also a correspondence of deformation families of pairs.  
Before we state the assertion,
we give a proof of
 the following Lemma.

\begin{lemma}\label{Picard_number_one}	
Let $X$ and $X'$ are irreducible symplectic manifold. Assume that
there exists a bimeromophic map $\phi : X \dasharrow X'$. We
also assume
$\dim H^{1,1}(X,\mathbb{Q})=1$ and
$q(\beta) \ge  0$ for every element $ \beta $ of $H^{1,1}(X,\mathbb{Q})$,
where $ q_{X} $ is the Beauville-Bogomolov quadratic form
on $ H^{2}(X,\mathbb{C}) $.
Then $X$ and $X'$ are isomorphic.
\end{lemma}
\begin{proof}
Since $ X $ and $ X' $ are bimeromorphic, we have an isomorphism
\[
\phi_{*} : H^{2}(X,\mathbb{C})  \cong H^{2}(X' , \mathbb{C})
\]
by Lemma \ref{isometry}.
Since $ \phi_{*} $ respects the Beauville-Bogomolov quadratics
and the Hodge structures, $ \dim H^{1,1}(X', \mathbb{Q}) = 1$ and
$ H^{1,1} (X',\mathbb{Q})$ is generated by a class $ \gamma \in H^{1,1}(X',\mathbb{Q}) $
such that $ q_{X'}(\gamma) \ge 0$, where $ q_{X'} $ 
is the Beauville-Bogomolov quadratic form on $ H^{2}(X',\mathbb{C}) $.
Let $ \mathcal{C}_{X} $ and $ \mathcal{C}_{X'} $ be the positive cones in $ H^{1,1}(X,\mathbb{R}) $
and $ H^{1,1}(X',\mathbb{R}) $, respectively.
By \cite[Corollary 7.2]{MR1664696}, $ \mathcal{C}_{X} $ and $ \mathcal{C}_{X'} $
coincide with the K\"ahler cones of $ X $ and $ X' $, respectively.
Since $ \phi_{*} $ maps $ \mathcal{C}_{X} $ to $ \mathcal{C}_{X'} $,
$ \phi_{*}\alpha $ is K\"ahler for all K\"ahler class of $ H^{1,1}(X,\mathbb{R}) $.
By \cite[Corollary 3.3]{MR642659},
$ \phi $ can be extended to an isomorphism.
\end{proof}

Now
we can state a correspondence of deformation families of pairs.

\begin{prop}\label{Birationarity-of-Kuranishi}
  Let $X$ are irreducible symplectic manifolds and $L$ a line bundle on $ X $.
  We also let $X'$ are irreducible symplectic manifold and
$L'$ a line bundle on $X'$. 
We denote the universal
family of  deformations of the pair 
$(X,L)$ by $(\mathfrak{X}_{L},\mathfrak{L})$
and the
parametrizing space
by $\mathrm{Def}(X,L)$.
We also denote by 
the universal family of deformations of the pair $(X',L')$ by 
$(\mathfrak{X}'_{L'},\mathfrak{L}')$ and
the parameter space 
by 
$ \mathrm{Def}(X',L') $.
Assume that there exists 
a birational map $\phi : X\dasharrow X'$ 
such that  $\phi_* L \cong L'$ and $ q_{X}(L) \ge 0 $, where
$ q_{X} $ is the Beauville-Bogomolov quadratic form
on $ H^{2}(X,\mathbb{C}) $. Then we have the followings.
\begin{itemize}
\item[(1)] There exist open subsets $U_{L}$ of $\mathrm{Def}(X,L)$ 
and $ U'_{L'} $ of $ \mathrm{Def}(X',L') $ such that 
they satisfy the following diagram
\begin{equation*}\label{birational_diagram_2}
 \xymatrix{
 \mathfrak{X}_{L}\times_{\mathrm{Def}(X,L)}U_{L} 
 \ar[r]^{\cong} \ar[d] & 
 \mathfrak{X}'_{L'}\times_{\mathrm{Def}(X',L')}U'_{L'}  \ar[d] \\
 U_{L} \ar[r]_{\cong}^{\varphi} & U'_{L'} },
\end{equation*}
where $\varphi $ is the isomorphism in the diagram of the assertion (2)
of Proposition \ref{birational_kuranishi}.
Moreover $ \mathrm{Def}(X,L)\setminus U_{L} $ is contained in a union of countably 
 hypersurfaces of $ \mathrm{Def}(X,L) $ and
$ \mathrm{Def}(X,L)\setminus U'_{L'} $ is also contained in a union of countably
 hypersurfaces of $ \mathrm{Def}(X,L) $.
\item[(2)] For every point $ s\in U_{L} $, 
$ (\phi_{s})_{*}\mathfrak{L}_{s} \cong  \mathfrak{L}'_{s'}$, where 
$ s'  = \varphi (s)$ and
$ \phi_{s} $ is the restriction of the isomorphism 
$
\mathfrak{X}_{L}\times_{\mathrm{Def}(X,L)}U_{L} 
 \to
 \mathfrak{X}'_{L'}\times_{\mathrm{Def}(X',L')}U'_{L'} 
$
to the fibres $ \mathfrak{X}_{L,s} $ and $ \mathfrak{X'}_{L',s'} $.
\end{itemize}
\end{prop}

\begin{proof}
We use the same notation in the statements and 
the proof of Proposition \ref{birational_kuranishi}.
If $ U\cap \mathrm{Def}(X,L) \ne \emptyset $, 
then $ U_{L} := U\cap \mathrm{Def}(X,L) $
and $ U'_{L'} := U'\cap \mathrm{Def}(X',L') $ satisfies the assertion (1) and
every point $ s \in U_{L} $ satisfy the assertion of (2) because the restricted
isomorphism satisfies the assertion (3) of Proposition \ref{birational_kuranishi}.
Let $s$
be a point of $\mathrm{Def}(X,L)$ such that 
$\dim H^{1,1}(\mathfrak{X}_{s},\mathbb{Q})=1$, 
where $\mathfrak{X}_{s}$ is the 
fibre at $s$. We will prove that $ s \in U $.
Since $ U $ is dense and open,
there exists
 a small disk $ S $ of $ \mathrm{Def}(X) $  
 such that $ s \in S $ and
 $ S\setminus \{s\} \subset U$.
 We denote $ \varphi (s) $ by $ s' $  and $ \varphi (S) $ by $ S' $.
 If we consider the base changes $ \mathfrak{X} \to \mathrm{Def}(X)$ by
 $ S $ and $ \mathfrak{X'}\to \mathrm{Def}(X') $ by 
 $ S' $, 
 we obtain the following diagram:
\[
\xymatrix{
 \mathfrak{X}_{S\setminus \{s\}} \ar[d] \ar[r]^{\cong}
  & \mathfrak{X}'_{S'\setminus \{s'\}} \ar[d] \\
 S \setminus s\ar[r]^{\varphi}_{\cong} & S'\setminus \{s'\}
 },
\]
By \cite[Theorem 4.3]{MR1664696}, there exists a birational map
$\mathfrak{X}_{s} \dasharrow \mathfrak{X}'_{s'}$.
By the definition of the Beauville-Bogomolov quadratic form \cite[Page 772]{Beauville}, the function
\[
\mathrm{Def}(X,L) \ni s \mapsto q_{\mathfrak{X}_{s}}(\mathfrak{L}_{s}) \in \mathbb{Z}
\]
is constant, where $ q_{\mathfrak{X}_{s}} $ stands for the Beauville-Bogomolov quadratic form
on $ H^{2}(\mathfrak{X}_{s},\mathbb{C}) $. Thus we have
$ q_{X}(L) = q_{\mathfrak{X}_{s}}(\mathfrak{L}_{s}) \ge 0 $.
By Lemma \ref{Picard_number_one},
$\phi_{s}$ is an ismorphism. This implies that $ s \in U $.
\end{proof}

\section{Proof of Theorem }
We start with giving a numerical criterion of existence of Lagrangian
fibrations.

\begin{lemma}\label{numerical_property}
Let $X$ be an irreducible symplectic manifold and $L$  a line bundle on $X$.
The linear system $|L|$ defines a 
Lagrangian fibration over the projective space
if and only if $L$ is nef and  $L$ has the following property:
\begin{equation}\label{global_section}
\dim H^{0}(X,L^{\otimes k}) = \dim H^{0}(\mathbb{P}^{1/2\dim X}, \mathcal{O}(k))
\end{equation}
for every positive integer $k$.
\end{lemma}

\begin{proof}
If $|L|$ defines a Lagrangian fibration over the projective space, 
it is trivial that $L$ is nef and
the dimension of global sections of $L^{\otimes k}$
satisfies the equation (\ref{global_section})
by Definition \ref{rational_Lag_def}. Thus we prove that $|L|$ defines
a Lagrangian fibration under the assumption that $L$ is nef and $\dim H^{0}(X,L^{\otimes k})$ 
satisfy the equation (\ref{global_section}).
By the assumption, the linear system $|L|$ defines
a rational map $X \dasharrow \mathbb{P}^{1/2\dim X}$. Let $\nu : Y \to X$ be
a resolution of indeterminacy and $g : Y \to \mathbb{P}^{1/2\dim X}$ 
is the induced morphism. 
Comparing $\nu^{*}L$ and $g^{*}\mathcal{O}(1)$, we have
\[
 \nu^{*}L \cong g^{*}\mathcal{O}(1)+F,
\]
where $F$ is a $\nu$-exceptional divisor. By multiplying the both hand sides,
we have
\[
 k\nu^{*}L \cong g^{*}\mathcal{O}(k) + kF.
\]
If $F \ne 0$, then the above isomorphism and the equality (\ref{global_section})
implies that $L$ is not semiample.
By the assumption, 
 $L$ is nef. 
If $ L^{\dim X} \ne 0 $, then $ L $ is also big and 
$ \dim H^{0}(X,L^{\otimes k}) $ does not satisfy the 
equation (\ref{global_section}). Thus $ L^{\dim X} = 0 $.
By \cite[Theorem 4.7]{MR946237}, we obtain
\[
 c_{X}q(kL + \alpha)^{1/2\dim X} = (kL + \alpha)^{\dim X},
\]
 where $ q_{X} $ is the Beauville-Bogomolov quadratic form
 on $ H^{2}(X,\mathbb{C}) $,
 $ c_{X} $ is the positive constant of $ X $ and $ \alpha $
 is a K\"ahler class of $ H^{1,1}(X,\mathbb{C}) $.
 Comparing the degrees of both hand sides of the above equation,
 we obtain that the numerical Kodaira dimension $ \nu (L) $ is $ (1/2)\dim X $.
 By the equation (\ref{global_section}), the Kodaira dimension
 $ \kappa (L) $ is also equal to $ (1/2)\dim X $.
Since $ K_{X} $ is trivial, the equality $ \nu (L) = \kappa (L) $ impies
that $ L $ is semiample by
 \cite[Theorem 6.1]{kawamata-freeness} 
and \cite[Theorem 1.1]{fujino-freeness}.
Thus $F = 0$ and 
the linear system $|L|$ defines the morphism 
$
f: X \to \mathbb{P}^{1/2\dim X} $.
The linear system $|lL|$ defines 
a morphism 
\[
f_{l} : X \to \mathrm{Proj}\oplus_{m\ge 0} H^{0}(X,L^{\otimes ml}) \cong 
\mathbb{P}^{\begin{pmatrix}
n+l \\ 
n
\end{pmatrix} -1 }.
\]
This morphism has
 connected fibres if $ l $ is sufficiently large. By the above
 expression,
$ f_{l} $ is the composition of $f$ and the Veronese embedding. This implies that
$f$ has connected fibres. 
\end{proof}

We introduce a criterion which asserts locally freeness of
direct images of line bundles.

\begin{lemma}\label{nakayama_freeness}
 Let $ \pi : \mathfrak{X}_{S} \to S $ be a smooth morphism over
 a small disk $ S $ with the reference point $ o $.
 We also let $ \mathfrak{L}_{S} $ be a line bundle on 
 $ \mathfrak{X}_{S} $.
 Assume that $ \mathfrak{X}_{S} $ and $ \mathfrak{L}_{S} $
 satisfy the following conditions.
 \begin{itemize}
 \item[(1)] The canonical bundle of every fibre is trivial.
 \item[(2)] For every point $ t $ of $ S \setminus \{o\}$,
           the restriction $ \mathfrak{L}_{S,t} $ of $ \mathfrak{L}_{S} $
           to
           the fibre $ \mathfrak{X}_{S,t} $ at $ t $
           is semiample.
 \item[(3)] The restriction $ \mathfrak{L}_{S,o} $ 
            of $ \mathfrak{L}_{S} $ to
 			the fibre $ \mathfrak{X}_{S,o} $ at $ o $
 			is nef.
 \end{itemize}
  Then the higher direct images $ R^{q}\pi_{*}\mathfrak{L}^{\otimes k}_{S} $
  are locally free for all $ q \ge 0 $ and $ k \ge 1$. Moreover
  the morphisms
  \begin{equation}\label{base_change}
     R^{q}\pi_{*}\mathfrak{L}^{\otimes k}_{S}\otimes k(o) \to 
     H^{q}(\mathfrak{X}_{S,o}, \mathfrak{L}_{S,o})
  \end{equation}
  are isomorphic for all $ q \ge 0 $ and $ k \ge 1 $.
\end{lemma}

\begin{proof}
The first part is a special case of \cite[Corollary 3.14]{nakayama-freeness}.
By the criteria of cohomological flatness in \cite[page 134]{MR0463470},
if $ R^{q}\pi_{*}\mathfrak{L}^{\otimes k}_{S}$ is locally free and the morphism
(\ref{base_change}) is isomorphic, then the morphism
\[
     R^{q-1}\pi_{*}\mathfrak{L}^{\otimes k}_{S}\otimes k(o) \to 
     H^{q-1}(\mathfrak{X}_{S,o}, \mathfrak{L}_{S,o})
\]
is also isomorphic for every $ k \ge 1 $. If $ q \ge \dim \mathfrak{X}_{S,s} + 1 $,
the both hand sides of the morphism (\ref{base_change}) are zero.
By a reverse induction, we obtain the last part of the assertions of Lemma.
\end{proof}

We need one more lemma to prove Theorem \ref{main_result}.
\begin{lemma}\label{nefness_of_non_projective}
Let $ X $ be an irreducible symplectic manifold. Assume that
$ X $ is not projective. Then a line bundle $ L $ such that
$ q_{X}(L) = 0$ is nef, where $ q_{X} $ is the
Beauville-Bogomolov quadratic form on $ H^{2}(X,\mathbb{C}) $.
\end{lemma}

\begin{proof}
 Assume that $ L $ is not nef.
 By \cite[Theorem 7.1]{MR1664696}, there exists a line bundle $ M $ on $ X $
 such that $ q_{X}(M,L) < 0 $ and $ q_{X}(M,\alpha) \ge 0 $ for all K\"ahler
 class $ \alpha \in H^{2}(X,\mathbb{C})$. 
 If we choose a suitable rational number $ \lambda $, we have $ q_{X}(L + \lambda M) > 0 $.
 This implies that $ X $ is projective by \cite[Theorem 2]{MR1965365}. That
 is a contradiction.
\end{proof}

Now we prove that if $ V_{\mathrm{reg}} \ne \emptyset $ then $ V_{\mathrm{reg}} $ is 
a dence open subset of $ \mathrm{Def}(X,L) $.

\begin{lemma}\label{nef_locus}
We use the same notation as in Theorem \ref{main_result}.  If 
$V_{\mathrm{reg}} \ne \emptyset$,
$V_{\mathrm{reg}}$ is  dense  and open in $\mathrm{Def}(X,L)$. Moreover
 $ \mathrm{Def}(X,L)\setminus V_{\mathrm{reg}} $ is contained in 
a countably union of hypersurfaces of $ \mathrm{Def}(X,L) $.
\end{lemma}

\begin{proof}
Let $t$ be a point of $V_{\mathrm{reg}}$ 
and we denote by $\mathfrak{X}_{t}$ the fibre at $ t $
 and by $\mathfrak{L}_{t}$
the restriction of $\mathfrak{L}$ 
to $ \mathfrak{X}_{t} $.
First we prove that $ V_{\mathrm{reg}} $ is open.
By the difinition of $ V_{\mathrm{reg}} $ in Theorem \ref{main_result}, 
the linear system $|\mathfrak{L}_{t}|$
defines a Lagrangian fibration 
$f_{t} : \mathfrak{X}_{t} \to \mathbb{P}^{1/2\dim \mathfrak{X}_{t}}$.
Let us consider the Relay spectral sequence
\[
 E^{p,q}_{2} = H^{p}(\mathbb{P}^{1/2 \dim \mathfrak{X}_{t}}, R^{q}(f_{t})_{*}f_{t}^{*}\mathcal{O} (1)
  ))
 \Longrightarrow
 E^{p+q}= H^{p+q}(\mathfrak{X}_{t}, (f_{t})^{*}\mathcal{O}(1)).
\]
The edge sequence of the above spectral sequence is
\[
0 \to H^{1}(\mathbb{P}^{1/2\dim \mathfrak{X_{t}}}, \mathcal{O}(1)) \to 
H^1(\mathfrak{X}_{t},f_{t}^{*}\mathcal{O}(1))
  \to H^{0}(\mathbb{P}^{1/2\dim \mathfrak{X}_{t}}, 
  R^{1}(f_{t})_{*}\mathcal{O}_{\mathfrak{X}_{t}}\otimes \mathcal{O}(1))
\]
%
%
By \cite[Theorem 1.2]{higher-direct-image},
\[
 R^{1}(f_{t})_{*}\mathcal{O}_{\mathfrak{X}_{t}} \cong 
 \Omega^{1}_{\mathbb{P}^{1/2 \dim \mathfrak{X}_{t}}}.
\]
Since $H^{1}(\mathbb{P}^{1/2\dim \mathfrak{X}_{t}}, 
\Omega_{\mathbb{P}^{1/2\dim \mathfrak{X}_{t}}} (1))=0$,
we have $H^{1}(\mathfrak{X}_{t},f_{t}^{*}\mathcal{O}(1))= 
H^{1}(\mathfrak{X}_{t},\mathfrak{L}_{t}) = 
0$.
By \cite[Corollary III. 3.9]{MR0463470},
$\pi_{*}\mathfrak{L}$ is locally free in an open neighbourhood of $ t $
and the morphism
\[
 \pi_{*}\mathfrak{L}\otimes k(t) \to H^{0}(\mathfrak{X}_{t}, \mathfrak{L}_{t})
\]
is bijective.
Combining the fact that $ \mathfrak{L}_{t} $ is free,
\[
 \pi^{*}\pi_{*}\mathfrak{L} \to \mathfrak{L}
\]
is surjective
over an open neighborhood of $ t $. This implies that $ V_{\mathrm{reg}} $ is open.

Next we prove that 
$ \mathrm{Def}(X,L)\setminus V_{\mathrm{reg}} $ is contained in a union
of countably hypersurfaces of $ \mathrm{Def}(X,L) $.
Since a union of real codimension two subsets cannot separate two non-empty open 
subsets, this implies that $V_{\mathrm{reg}}$ is dense. 
 Let $t'$ be a point of the closure of $V_{\mathrm{reg}}$ such that 
 $\dim H^{1,1}(\mathfrak{X}_{t'} ,\mathbb{Q}) = 1$,
 where $ \mathfrak{X}_{t'} $ is the fibre at $ t' $. We denote by
 $ \mathfrak{L}_{t'} $ the restriction of $ \mathfrak{L} $ to $ \mathfrak{X}_{t'} $.
 By the definition of the Beauville-Bogomolov quadratic in \cite[page 772]{Beauville},
 the function
 \[
  \mathrm{Def}(X,L)\ni t \mapsto q_{\mathfrak{X}_{t}}(\mathfrak{L}_{t}) \in \mathbb{Z}
 \]
 is a constant function, where $ q_{\mathfrak{X}_{t}} $ is the Beauville-Bogomolov
 quadratic form on $ H^{2}(\mathfrak{X}_{t},\mathbb{C}) $.
 Thus $ q_{\mathfrak{X}_{t'}} (\mathfrak{L}_{t'}) = 0$.
 Since $ H^{1,1}(\mathfrak{X}_{t'},\mathbb{Q}) $ is spanned by $ \mathfrak{L}_{t'} $,
 $ \mathfrak{X}_{t'} $ is not projective by \cite[Theorem 2]{MR1965365}. Thus
  $ \mathfrak{L}_{t'} $ is nef by Lemma \ref{nefness_of_non_projective}.
We choose a small disk $S$ in $\mathrm{Def}(X,L)$ such that $ t' \in S $ and
$S\setminus \{t'\} \subset V_{\mathrm{reg}}$. We also
consider the restriction
family $\pi_S :\mathfrak{X}_L\times_{\mathrm{Def}(X,L)}S \to S$. 
Then $\mathfrak{L}^{\otimes k}_{t"}$ is free for  every point $t"$ of $S\setminus \{t'\}$
and $ k \ge 1 $, where $ \mathfrak{L}_{t"} $ is the restriction of $ \mathfrak{L} $ to
the fibre $ \mathfrak{X}_{t"} $ at $ t" $.
By Lemma \ref{nakayama_freeness},
\( (\pi_{S})_{*} \mathfrak{L}^{\otimes k}  
 \)
 is locally free and
the morphism
\[
 (\pi_{S})_{*}\mathfrak{L}^{\otimes k}\otimes k(t') \to 
 H^{0}(\mathfrak{X}_{t'}, \mathfrak{L}_{t'}^{\otimes k})
\]
is bijective for every $ k \ge 1$. By Lemma \ref{numerical_property},
$t' \in V_{\mathrm{reg}}$. 
Let $ W $ be the subset  of $ \mathrm{Def}(X,L) $ defined by
\[
W := \{
 t \in \mathrm{Def}(X,L); \dim H^{1,1}(\mathfrak{X}_{t},\mathbb{Q}) \ge 2
\}.
\]
By the above argument,
 $ \mathrm{Def}(X,L)\setminus V_{\mathrm{reg}} \subset W$.
By \cite[(1.14)]{MR1664696}, $ W $
is contained in a union of countably hypersurfaces of $ \mathrm{Def}(X,L)$
and we are done.
%
\end{proof}

We give a proof of Theorem \ref{main_result}.

\begin{proof}[Proof of Theorem \ref{main_result}]
The proof consists of three parts. We start with proving the following Claim.
\begin{claim}\label{first_step_of_proof_of_Theorem}
If $ V \ne \emptyset$, then $ V_{\mathrm{reg}} \ne \emptyset $.
\end{claim}
\begin{proof}
We may assume that the reference point $ o $ of $ \mathrm{Def}(X,L) $ is contained
in $ V $. 
By Definition \ref{rational_Lag_def},
there exists a birational map
$\phi : X \dasharrow X'$
such that the linear system $|\phi_{*}L|$ 
defines a Lagrangian fibration $X' \to 
\mathbb{P}^{1/2\dim X}$. 
Let $ L' := \phi_{*}L $ and
$(\mathfrak{X}'_{L'}, \mathfrak{L}')$ be the universal family of
deformations of
the pair $ (X',L') $.
Let $ V'_{\mathrm{reg}} $ be the locus of Lagrangian fibration
of $ \mathrm{Def}(X',L') $. Then the reference point $ o' $  of 
$ \mathrm{Def}(X',L') $ is contained in $ V'_{\mathrm{reg}} $.
By Lemma \ref{nef_locus}, $V'_{\mathrm{reg}}$ is a dense open set of
$ \mathrm{Def}(X',L') $.
By Proposition \ref{Birationarity-of-Kuranishi},
we also have  dense open sets 
$U'_{L'}$ 
of $\mathrm{Def}(X',L')$
and $ U_{L} $ of $ \mathrm{Def}(X,L) $ which satisfy the
following diagram:
\[
 \xymatrix{
 \mathfrak{X}_{L}\times_{\mathrm{Def}(X,L)} U_{L} \ar[r]^{\cong} \ar[d] &
 \mathfrak{X}'_{L'}\times_{\mathrm{Def}(X',L')} U_{L'}  \ar[d] \\
 U_{L} \ar[r]_{\cong}^{\varphi} & U'_{L'}
 }
\]
By the assertion (2) of Proposition \ref{Birationarity-of-Kuranishi},
$ \varphi^{-1} ( U'_{L'} \cap V'_{\mathrm{reg}} ) \subset V_{\mathrm{reg}}$. 
Since $ U'_{L'} \cap V'_{\mathrm{reg}} \ne \emptyset$, we obtain
$ V_{\mathrm{reg}} \ne \emptyset$.
\end{proof}
By Claim \ref{first_step_of_proof_of_Theorem} and Lemma \ref{nef_locus},
$ \mathrm{Def}(X,L) $ coinsides with the closure of $ V_{\mathrm{reg}} $ 
under the assumption that $ V \ne \emptyset $. 
\begin{claim}\label{the_second_step_of_the_Proof}
Assume that the reference point $ o $ of $ \mathrm{Def}(X,L) $
 is contained in the closure of $ V_{\mathrm{reg}} $ and $ L $ is nef. Then
 $ o \in V_{\mathrm{reg}} $.
\end{claim}

\begin{proof}
By the assumption that $ o \in \overline{V}_{\mathrm{reg}} $, we choose a small disk $ S $ in
$ \mathrm{Def}(X,L) $
which has the following properties:
\begin{itemize}
\item[(1)] $ o \in S $.
\item[(2)] $ S\setminus \{o\} \subset V_{\mathrm{reg}} $.
\end{itemize}
Let $ \pi_{S} : \mathfrak{X}_{L}\times_{\mathrm{De}(X,L)}S \to S $ be
the restriction family
and $ \mathfrak{L}_{S} $ the restriction of the universal bundle $ \mathfrak{L} $ to 
$ \mathfrak{X}_{L} \times_{\mathrm{Def}(X,L)} S $. Then 
$ \pi_{S} : \mathfrak{X}_{L}\times_{\mathrm{Def}(X,L)} S \to S$ 
and $ \mathfrak{L}_{S} $
satisfy all assumptions of Lemma \ref{nakayama_freeness}. Hense
$ \pi_{*} \mathfrak{L}_{S}^{\otimes k}$ are locally free and
the morphisms
\[
  \pi_{*} \mathfrak{L}_{S}^{\otimes k}\otimes k(s) \to 
  H^{0}(\mathfrak{X}_{s}, \mathfrak{L}_{s}^{\otimes k})
\]
are isomorphic for all $ k \ge 0 $ and all points $ s \in S $.
This implies that the pair $ (X,L) $ satisfies the all assumptions of Lemma \ref{numerical_property}
and we obtain
$ o \in V_{\mathrm{reg}} $.
\end{proof}

\begin{claim}\label{the_third_step_of_the_Proof}
Assume that the reference point $ o $ of $ \mathrm{Def}(X,L) $
 is contained in the closure of $ V_{\mathrm{reg}} $ and $ c_{1}(L) $ belongs to
 the birational K\"ahler cone. Then
 $ o \in V $.
\end{claim}

\begin{proof}
We remark that
$ X $ is projective by Lemma \ref{nefness_of_non_projective}.
We consider the same restriction family 
$ \pi : \mathfrak{X}_{L}\times_{\mathrm{Def}(X,L)}  S \to S$
in the proof of Claim \ref{the_second_step_of_the_Proof}.
By the upper semicontinuity of the function
\[
 s \in S \mapsto \dim H^{0}(\mathfrak{X}_{s},\mathfrak{L}_{s}),
\]
$ \mathfrak{L}_{o} = L $ is effective. 
By \cite[Theorem 1.2]{0907.5311}, 
there exists a birational map
\( \phi: X \dasharrow X' \) such that $ L' $ is nef, where
$ L' = \phi_{*}L$. 
By Proposition \ref{Birationarity-of-Kuranishi},
we have dense open sets 
$U'_{L'}$ 
of $\mathrm{Def}(X',L')$
and $ U_{L} $ of $ \mathrm{Def}(X,L) $ which satisfy the
following diagram:
\[
 \xymatrix{
 \mathfrak{X}_{L}\times_{\mathrm{Def}(X,L)} U_{L} \ar[r]^{\cong} \ar[d] &
 \mathfrak{X}'_{L'}\times_{\mathrm{Def}(X',L')} U_{L'}  \ar[d] \\
 U_{L} \ar[r]_{ \cong}^{\varphi} & U'_{L'}
 }
\]
Let $ V'_{\mathrm{reg}}$ be the locus of Lagrangian fibrations of 
$ \mathrm{Def}(X',L') $. Then $ V'_{\mathrm{reg}} \ne \emptyset$ because
the image  $\varphi( V_{\mathrm{reg}} \cap U_{L})$ is contained
in $ V'_{\mathrm{reg}} $ by Proposition \ref{Birationarity-of-Kuranishi} (2).
By Lemma \ref{nef_locus}, $ V'_{\mathrm{reg}} $ is dense. Hence the reference
point $ o' $ of $ \mathrm{Def}(X',L') $ is contained in the closure of 
$ V'_{\mathrm{reg}} $. By Claim \ref{the_second_step_of_the_Proof},
 $ o' \in V'_{\mathrm{reg}} $. This implies that $ o\in V $.
\end{proof}

We finish the proof of Theorem \ref{main_result}. If $ V \ne \emptyset $, $ V_{\mathrm{reg}} $
is open and dense in $ \mathrm{Def}(X,L) $ by Claim \ref{first_step_of_proof_of_Theorem}.
Thus every point $ s $ of $ \mathrm{Def}(X,L)_{\mathrm{mov}} $ 
is contained in the closure of $ V_{\mathrm{reg}} $.
Then $ s \in V $ by Claim \ref{the_third_step_of_the_Proof} and we are done.
\end{proof}

\begin{proof}[Proof of Corollary \ref{application}]
 We use the same notation of Theorem \ref{main_result} and Corollary \ref{application}.
 We also define 
 the subset $ \Lambda $ of $ \mathrm{Def}(X) $ by
\begin{eqnarray*}
   \Lambda := \{
   s \in \mathrm{Def}(X) ; &\mathfrak{X}_{s}& \mbox{is isomorphic to
   the $ n $-pointed
   Hilbert scheme of $ K3 $ }\\ 
   && \mbox{or a generalized Kummer variety }\}
\end{eqnarray*}
  If  $ \Lambda \cap \mathrm{Def}(X,L)_{\mathrm{mov}} \ne \emptyset$,
  then the restriction of the universal bundle $ \mathfrak{L}_{s} $ to the
  fibre $ \mathfrak{X}_{s} $ at $ s \in \Lambda\cap \mathrm{Def}(X,L)_{\mathrm{mov}} $
 defines a rational Lagrangian fibration by \cite[Conjecture 1.4, Theorem 1.5]{1301.6968}, \cite[Theorem 1.3]{1301.6584} and
 \cite[Proposition 3.36]{1206.4838}.
 First we prove that $ \Lambda \cap \mathrm{Def}(X,L) $ is dense in $ \mathrm{Def}(X,L) $.
 The subset $ \Lambda $ 
 is dense in $ \mathrm{Def}(X) $ by
  \cite[Theorem 1.1, Theorem 4.1]{2012arXiv1201.0031M}. Moreover,
 by \cite[Th\'eor\`em 6, 7]{Beauville}, each irreducible component of
 $ \Lambda $ 
 forms a smooth hypersurface of $ \mathrm{Def}(X) $. Therefore 
 $ \Lambda \cap \mathrm{Def}(X,L)$ is dense in $ \mathrm{Def}(X,L) $.
 Next we will prove the following Lemma.
 \begin{lemma}\label{density_of_movable}
 Under the same assumptions and notation of Theorem \ref{main_result},
 the closure of $\mathrm{Def}(X,L)\setminus \mathrm{Def}(X,L)_{\mathrm{mov}} $ is a proper closed subset
 of $ \mathrm{Def}(X,L) $.
 \end{lemma}
 \begin{proof}
  We derive a contradiction assuming that the closure of 
  $ \mathrm{Def}(X,L)\setminus \mathrm{Def}(X,L)_{\mathrm{mov}} $ coincides with
  $ \mathrm{Def}(X,L) $. For a point $ s \in \mathrm{Def}(X,L)\setminus \mathrm{Def}(X,L)_{\mathrm{mov}} $,
  we denote by $ \mathfrak{L}_{s} $ the restriction of the universal bundle $ \mathfrak{L} $ 
  to the fibre $ \mathfrak{X}_{s} $ at $ s $. We will prove that $ \mathfrak{L}_{s} $ is big.
  By Corollary \cite[Corollary 3.10]{MR1664696}, 
  the interior of the positive cone of an irreducible symplectic manifold is contained
  in the effective cone. By the assumption,
   $ L $ belongs to the closure of the positive cone of $ X $.
   Hence
  $ \mathfrak{L}_{s} $ also belongs to the closure of the positive cone of $ \mathfrak{X}_{s} $.
  Thus $ \mathfrak{L}_{s} $ is pseudo-effective. 
  By \cite[Theorem 3.1]{0907.5311},
  We obtain the $ q $-Zariski decomposition
  \[
  \mathfrak{L}_{s} = P_{s} + N_{s}
  \]
  By \cite[Theorem 3.1 (I) (iii)]{0907.5311}
  \[
 0 =  q_{\mathfrak{X}_{s}}(\mathfrak{L}_{s}) = q_{\mathfrak{X}_{s}}(P_{s}+N_{s}) = 
  q_{\mathfrak{X}_{s}}(P_{s}) + q_{\mathfrak{X}_{s}}(N_{s}).
  \]
  Since $ \mathfrak{L}_{s} $ does not belongs to the birational K\"ahler cone of $ \mathfrak{X}_{s} $,
  $ N \ne 0 $. This implies that $ q_{\mathfrak{X}_{s}}(P_{s}) > 0 $.
  We deduce $ P_{s} $  is big by \cite[Corollary 3.10]{MR1664696}.
  Hence $ \mathfrak{L}_{s} $ is also big.

  Let us consider the following function
  \[
  \mathrm{Def}(X,L) \ni s \mapsto h_{n}(s):= \dim H^{0}(\mathfrak{X}_{s},\mathfrak{L}_{s}^{n}) \in \mathbb{Z}.
  \]
  By the upper semicontinuity of $ h_{n}(s) $, there exists an open set $ W $ of $ \mathrm{Def}(X,L) $
  such that for every point $ s $ of $ W $, 
  \[
  h_{n}(t) \ge h_{n}(s)
  \]
  for all points $ t \in \mathrm{Def}(X,L) $. By the assumption
  that the closure of $ \mathrm{Def}(X,L)\setminus \mathrm{Def}(X,L)_{\mathrm{mov}} $
  coincides with $ \mathrm{Def}(X,L) $, 
  $ W \cap (\mathrm{Def}(X,L)\setminus \mathrm{Def}(X,L)_{\mathrm{mov}}) \ne \emptyset$.
  In the first half of the proof of this Lemma, we had proved that
  $ \mathfrak{L}_{s} $ is big for every point 
  $ s\in \mathrm{Def}(X,L)\setminus \mathrm{Def}(X,L)_{\mathrm{mov}}$.
  This implies that $ \mathfrak{L}_{t} $ is big for every point $ \mathrm{Def}(X,L) $. 
  Let $ t $ be a point of $ \mathrm{Def}(X,L) $ such 
  that $ \dim H^{1,1}(X,\mathbb{Q}) = 1 $. Then
   $ \mathfrak{L}_{t} $
  is nef  by Lemma \ref{nefness_of_non_projective}.
  Since $ \mathfrak{L}_{t} $ is nef and big, the higher cohomologies of $ \mathfrak{L}_{t} $ vanish
  By the Riemann-Roch formula in \cite[(1.11)]{MR1664696},
  we obtain
  \[
   \dim H^{0}(\mathfrak{X}_{t},\mathfrak{L}_{t}^{m}) =
   \sum_{j=0}^{\dim \mathfrak{X}_{t}/2}\frac{a_{j}}{2j}m^{2j}q_{\mathfrak{X}_{t}}(\mathfrak{L}_{t})^{j}
   = \chi (\mathcal{O}_{\mathfrak{X}_{t}}),
  \]
  because $ q_{\mathfrak{X}_{t}}(\mathfrak{L}_{t}) = q_{X}(L)= 0 $.
  That is a contradiction.
 \end{proof}
 We finish the proof of Corollary \ref{application}.
 If $ \Lambda\cap \mathrm{Def}(X,L)_{\mathrm{mov}} = \emptyset$,
 $ \mathrm{Def}(X,L) \setminus \mathrm{Def}(X,L)_{\mathrm{mov}} $ contains dense subsets
 of $ \mathrm{Def}(X,L) $.
 This contradicts Lemma \ref{density_of_movable}.
\end{proof}

\bibliographystyle{halpha}

\bibliography{semiample.bib}

\begin{thebibliography}{Huy03b}

\bibitem[AC08]{MR2400885}
Ekaterina Amerik and Fr{\'e}d{\'e}ric Campana.
\newblock Fibrations m\'eromorphes sur certaines vari\'et\'es \`a fibr\'e
  canonique trivial.
\newblock {\em Pure Appl. Math. Q.}, 4(2, part 1):509--545, 2008.

\bibitem[Bea83]{Beauville}
Arnaud Beauville.
\newblock Vari\'et\'es {K}\"ahleriennes dont la premi\`ere classe de {C}hern
  est nulle.
\newblock {\em J. Differential Geom.}, 18(4):755--782 (1984), 1983.

\bibitem[Bea85]{MR785234}
Arnaud Beauville.
\newblock Vari\'et\'es k\"ahl\'eriennes compactes avec {$c_1=0$}.
\newblock {\em Ast\'erisque}, (126):181--192, 1985.
\newblock Geometry of $K3$ surfaces: moduli and periods (Palaiseau, 1981/1982).

\bibitem[BM13]{1301.6968}
Arend Bayer and Emanuele Macri.
\newblock Mmp for moduli of sheaves on k3s via wall-crossing: nef and movable
  cones, lagrangian fibrations, 2013, arXiv:1301.6968.

\bibitem[BS76]{MR0463470}
Constantin B{\u{a}}nic{\u{a}} and Octavian St{\u{a}}n{\u{a}}{\c{s}}il{\u{a}}.
\newblock {\em Algebraic methods in the global theory of complex spaces}.
\newblock Editura Academiei, Bucharest; John Wiley \& Sons, London-New
  York-Sydney, 1976.
\newblock Translated from the Romanian.

\bibitem[COP10]{MR2739808}
Fr{\'e}d{\'e}ric Campana, Keiji Oguiso, and Thomas Peternell.
\newblock Non-algebraic hyperk\"ahler manifolds.
\newblock {\em J. Differential Geom.}, 85(3):397--424, 2010.

\bibitem[Fuj81]{MR642659}
Akira Fujiki.
\newblock A theorem on bimeromorphic maps of {K}\"ahler manifolds and its
  applications.
\newblock {\em Publ. Res. Inst. Math. Sci.}, 17(2):735--754, 1981.

\bibitem[Fuj87]{MR946237}
Akira Fujiki.
\newblock On the de {R}ham cohomology group of a compact {K}\"ahler symplectic
  manifold.
\newblock In {\em Algebraic geometry, {S}endai, 1985}, volume~10 of {\em Adv.
  Stud. Pure Math.}, pages 105--165. North-Holland, Amsterdam, 1987.

\bibitem[Fuj11]{fujino-freeness}
Osamu Fujino.
\newblock On {K}awamata's theorem.
\newblock In {\em Classification of algebraic varieties}, EMS Ser. Congr. Rep.,
  pages 305--315. Eur. Math. Soc., Z\"urich, 2011.

\bibitem[Huy99]{MR1664696}
Daniel Huybrechts.
\newblock Compact hyper-{K}\"ahler manifolds: basic results.
\newblock {\em Invent. Math.}, 135(1):63--113, 1999.

\bibitem[Huy03a]{MR1965365}
Daniel Huybrechts.
\newblock Erratum: ``{C}ompact hyper-{K}\"ahler manifolds: basic results''
  [{I}nvent. {M}ath. {\bf 135} (1999), no. 1, 63--113; {MR}1664696
  (2000a:32039)].
\newblock {\em Invent. Math.}, 152(1):209--212, 2003.

\bibitem[Huy03b]{MR1992275}
Daniel Huybrechts.
\newblock The {K}\"ahler cone of a compact hyperk\"ahler manifold.
\newblock {\em Math. Ann.}, 326(3):499--513, 2003.

\bibitem[Hwa08]{MR2453602}
Jun-Muk Hwang.
\newblock Base manifolds for fibrations of projective irreducible symplectic
  manifolds.
\newblock {\em Invent. Math.}, 174(3):625--644, 2008.

\bibitem[Kaw85]{kawamata-freeness}
Y.~Kawamata.
\newblock Pluricanonical systems on minimal algebraic varieties.
\newblock {\em Invent. Math.}, 79(3):567--588, 1985.

\bibitem[KV12]{1208.4626}
Ljudmila Kamenova and Misha Verbitsky.
\newblock Families of lagrangian fibrations on hyperkaehler manifolds, 2012,
  arXiv:1208.4626.

\bibitem[Mar13]{1301.6584}
Eyal Markman.
\newblock Lagrangian fibrations of holomorphic-symplectic varieties of
  $k3^[n]$-type, 2013, arXiv:1301.6584.

\bibitem[Mat99]{addendum_fibrespace}
Daisuke Matsushita.
\newblock On fibre space structures of a projective irreducible symplectic
  manifold.
\newblock {\em Topology}, 38(1):79--83, 1999.

\bibitem[Mat01]{fibrespace}
Daisuke Matsushita.
\newblock Addendum: ``{O}n fibre space structures of a projective irreducible
  symplectic manifold'' [{T}opology {\bf 38} (1999), no.\ 1, 79--83;
  {MR}1644091 (99f:14054)].
\newblock {\em Topology}, 40(2):431--432, 2001.

\bibitem[Mat05]{higher-direct-image}
Daisuke Matsushita.
\newblock Higher direct images of dualizing sheaves of {L}agrangian fibrations.
\newblock {\em Amer. J. Math.}, 127(2):243--259, 2005.

\bibitem[Mat08]{MR2357635}
Daisuke Matsushita.
\newblock On nef reductions of projective irreducible symplectic manifolds.
\newblock {\em Math. Z.}, 258(2):267--270, 2008.

\bibitem[MM12]{2012arXiv1201.0031M}
E.~{Markman} and S.~{Mehrotra}.
\newblock {Hilbert schemes of K3 surfaces are dense in moduli}.
\newblock {\em ArXiv e-prints}, December 2012, 1201.0031.

\bibitem[MZ09]{0907.5311}
Daisuke Matsushita and De-Qi Zhang.
\newblock Zariski f-decomposition and lagrangian fibration on hyperkaehler
  manifolds.
\newblock Mathematical Research Letters 20 (2013) 951 - 959, 2009,
  arXiv:0907.5311.

\bibitem[Nak87]{nakayama-freeness}
Noboru Nakayama.
\newblock The lower semicontinuity of the plurigenera of complex varieties.
\newblock In {\em Algebraic geometry, {S}endai, 1985}, volume~10 of {\em Adv.
  Stud. Pure Math.}, pages 551--590. North-Holland, Amsterdam, 1987.

\bibitem[Ver10]{MR2585581}
Misha Verbitsky.
\newblock Hyper{K}\"ahler {SYZ} conjecture and semipositive line bundles.
\newblock {\em Geom. Funct. Anal.}, 19(5):1481--1493, 2010.

\bibitem[Yos12]{1206.4838}
Kota Yoshioka.
\newblock Bridgeland's stability and the positive cone of the moduli spaces of
  stable objects on an abelian surface, 2012, arXiv:1206.4838.

\end{thebibliography}

\end{document}